\documentclass[amstex,10pt,reqno]{amsart}
\usepackage{amsmath,amsfonts,amssymb,amsthm,enumerate,hyperref,multicol,enumitem}

\newtheorem{lemma}{Lemma}
\newtheorem{thm}{Theorem}

\newtheorem{prop}{Proposition}
\newtheorem{definition}{Definition}

\newtheorem{remark}{Remark}
\numberwithin{equation}{section}
\allowdisplaybreaks

\begin{document}
\title[Study of fractional evolution equations involving Hilfer fractional derivative of order $1<\gamma<2$ and type $0 \leq \delta \leq 1$]
{Study of fractional evolution equations involving Hilfer fractional derivative of order $1<\gamma<2$ and type $0 \leq \delta \leq 1$}


\author{}
\maketitle

\leftline{ \scriptsize \it  }

\begin{center}
\textbf{Anjali Jaiswal$^1$, D. Bahuguna$^2$} \vskip0.2in
$^{1,2}$Department of Mathematics\\[0pt]
Indian Institute of Technology Kanpur\\[0pt]
Kanpur-208016, India\\[0pt]
\end{center}

\begin{abstract}
In this paper we investigate  fractional differential equations with Hilfer fractional derivative of order $1<\gamma<2$ and type $\delta \in [0,1]$ in a Banach space. We introduce a family of general fractional cosine operator functions of order $1<\gamma<2$ and type $\delta \in [0,1]$ and discuss their properties to give a suitable definition of mild solution of the semilinear evolutin equation. In last section an example is presented.  \newline
\end{abstract}

\section{Introduction}
Fractional calculus has been appeared as an efficient tool for analyzing various real problems due to its nonlocal behavior and linearity, this include fractional oscillator, viscoelastic models, diffusion in porous media, signal analysis, complex systems, medical imaging, pollution control and population growth etc. Several definitions for the fractional derivatives and integrals are present in the literature. The mostly used are Riemann-Liouville, Caputo and Grunwald-Letnikov fractional derivative etc. Later, a new notion of fractional derivative was introduced named as ``Generalized fractional derivative of order $\gamma$ and type $\delta \in [0,1]$", which was utilised in the theoretical modelling of broadband dielectric relaxation spectroscopy for glasses. It contains  Riemann-Liouville and Caputo fractional derivative as  special cases for $\delta=0$ and $\delta=1$ respectively. Some important research papers containing Hilfer fractional derivative are-\cite{sandev2011fractional},\cite{furati2012existence}, \cite{kamocki2016fractional},\cite{mahmudov2015approximate}, \cite{ahmed2018impulsive},\cite{gu2015existence},\cite{ahmed2018nonlocal}, \cite{mei2015existence},\cite{hilfer2002experimental}. In most of the cases $0<\gamma<1$.\\

\indent Researchers have been trying many problems modelled as differential equations of arbitrary order. In most of the scenarios the order  of the fractional derivative is taken less than 1.   Higher order problems have been discussed in less quantity but few papers are available with order of derivative $\gamma \in (1,2)$.  In \cite{MPJ}, Mei et al. studied the properties of Mittag-Leffler function $E_{\gamma}(at^{\gamma})$ for $1<\gamma<2$ and defined a function  named $\gamma$-order cosine function having the similar properties as $E_{\gamma}(at^{\gamma})$ and proved that it is correlated with the solution operator of an abstract Cauchy problem. The Cauchy problem is well-posed iff the linear operator generates an $\gamma$-order cosine function. In \cite{MPZ}, Mei et al. examined a $\gamma$-order Cauchy problem with Riemann-Liouvile fractional derivative, $1<\gamma<2$, by defining a new family of linear operators called as ``Riemann-Liouville $\gamma$-order fractional resolvent".\\

\indent Shu et al. \cite{XQW} investigated a  semilinear fractional integro-differential equation of order $1<\gamma<2$
 whose associated linear operator is a sectorial operator of type $(M,\theta, \gamma, \mu)$. They defined some families of operators to give a suitable definition of a mild solution  and established the existence results for a mild solution using the aid of Krasnoselskii's fixed point and contraction mapping theorems. Similar type of problems has been discussed in \cite{WS},\cite{ZLW},\cite{G1},\cite{G2},\cite{G3},\cite{G4},\cite{G5}. \\

Mei et al. \cite{ZDM} considered a FDE with Hilfer derivative of order $1<\gamma<2$ and type $\delta\in [0,1]$ formulated as
\begin{equation}\label{ch6intro1}
D_{0}^{\gamma,\delta}\omega(t)=\mathcal{A}\omega(t),\; t>0
\end{equation}
\begin{equation}\label{ch6intro2}
(g^{(1-\delta)(2-\gamma)}*\omega)(0)=0, (g^{(1-\delta)(2-\gamma)}*\omega)^{'}(0)=y,
\end{equation}
where $\mathcal{A}: D(\mathcal{A}) \subset X \rightarrow X$ is a closed and densely defined linear operator on a Banach space $X$ and $D_{0}^{\gamma,\delta}\omega$ is the Hilfer fractional derivative of order $1<\gamma<2$ and type $\delta \in [0,1]$. They introduced a new family of bounded linear operators ``general fractional sine function" of order $\gamma\in (1,2)$ and type $\delta \in [0,1]$. They showed that these operators are essentially equivalent to a ``general fractional resolvent operator" and used the developed operatic approach  to establish the well-posedness of the \eqref{ch6intro1}-\eqref{ch6intro2}. \\

\indent Motivated by these articles, we consider the following linear
\begin{equation}\label{eq1.5}
D_{0}^{\gamma,\delta}\omega(t)=\mathcal{A}\omega(t),
\end{equation}
\begin{equation}\label{eq1.6}
(g^{(1-\delta)(2-\gamma)}*\omega)(0)=\omega_{1}, (g^{(1-\delta)(2-\gamma)}*\omega)^{'}(0)=\omega_{2},
\end{equation}
 and semilinear fractional differential equation
\begin{equation}\label{eq1.7}
D_{0}^{\gamma,\delta}\omega(t)=\mathcal{A}\omega(t)+\eta(t,\omega(t))
\end{equation}
\begin{equation}\label{eq1.8}
(g^{(1-\delta)(2-\gamma)}*\omega)(0)=\omega_{1}, (g^{(1-\delta)(2-\gamma)}*\omega)^{'}(0)=\omega_{2},
\end{equation}
in the Banach space $X$, where $\mathcal{A}$ is a closed linear operator on $X$.\\

In this paper, we firstly consider a linear fractional differential equation in $\mathbb{R}$ and discuss about the existence of solution. Later we will generalize the properties of $t^{\gamma+\delta(2-\gamma)-2}E_{\gamma,\gamma+\delta(2-\gamma)-1}(ct^{\gamma})$ to give  definition of a family of bounded linear operators, that will be used to study the problems \eqref{eq1.5},\eqref{eq1.6} and \eqref{eq1.7},\eqref{eq1.8}.

\section{Preliminaries}
In this section, we recall some basic definitions related to fractional derivatives and integrals and some results of measure of non-compactness.
\begin{definition}  The Riemann-Liouville fractional integral of order $\gamma>0$ is defined  by
\begin{equation}
J_{t}^{\gamma}\omega(t)=\frac{1}{\Gamma(\gamma)}\int_{0}^{t}(t-s)^{\gamma-1}\omega(s)ds,\; t>0.
\end{equation}
\end{definition}
\begin{definition}
Let $\gamma>0, m=[\gamma]$ and $ \delta \in [0,1]$. The Hilfer fractional derivative of order $\gamma$ and type $\delta$ is defined by
\begin{equation}
D_{0}^{\gamma,\delta}\omega(t)=J_{t}^{\delta(m-\gamma)}\frac{d^{m}}{dt^{m}}J_{t}^{(1-\delta)(m-\gamma)}\omega(t),\; t>0.
\end{equation}
For $\delta=0$, it gives the Riemann-Liouville fractional derivative and for $\delta=1$, it gives the Caputo fractional derivative.
\end{definition}
\subsection{Measure of Non-compactness}
For any bounded subset $\mathfrak{B}$ of the Banach space $X$,  the Hausdorff measure of non-compactness $\psi$ is defined by
\begin{equation}
	\psi(\mathfrak{B})=\inf\bigg\{\epsilon>0: \mathfrak{B}\subset \bigcup_{j=1}^{k}
	\mathcal{B}_{\epsilon}(x_{j})\: \mbox{where}\: x_{j} \in X, k \in
	\mathbb{N}\bigg\},
\end{equation}
where $\mathcal{B}_{\epsilon}(x_{j})$ is a ball of radius $\leq \epsilon$ centered
at $x_{j}, j=1,2,\ldots,k$.\\

The  another measure of noncompactness $\phi$ by introduced Kurtawoski for bounded subset $\mathfrak{B}$
of $X$ is given as
\begin{equation}
	\phi(\mathfrak{B})=\inf\bigg\{\epsilon>0: \mathfrak{B}\subset \bigcup_{j=1}^{k} M_{j}\:
	\mbox{and}\: diam(M_{j})\leq \epsilon \bigg\},
\end{equation}
where $diam(M_{j})$ is the diameter of $M_{j}$ .\\

The well known properties are:-
\begin{enumerate}[label=(\roman*)]
\item  $\psi(\mathfrak{B})=0 \Longleftrightarrow \mathfrak{B}$ is relatively compact in $X$;
	\item For any bounded subsets $\mathfrak{B}_{1}, \mathfrak{B}_{2}$ of $X$, $\psi(\mathfrak{B}_{1}) \leq \psi(\mathfrak{B}_{2})$ if $\mathfrak{B}_{1} \subset
	\mathfrak{B}_{2}$;
	\item $\psi(\{x\}\cup \mathfrak{B})=\psi(\mathfrak{B})$ ; 
	\item $\psi(\mathfrak{B}_{1}\cup \mathfrak{B}_{2}) \leq \max\{ \psi(\mathfrak{B}_{1}),\psi(\mathfrak{B}_{2})\} $;
	\item $\psi(\mathfrak{B}_{1}+\mathfrak{B}_{2}) \leq \psi(\mathfrak{B}_{1})+\psi(\mathfrak{B}_{2}) $, where
	$\mathfrak{B}_{1}+\mathfrak{B}_{2}=\{x+y:x\in \mathfrak{B}_{1}, y\in \mathfrak{B}_{2} \}$;
	\item $\psi(r \mathfrak{B}) \leq |r|\psi(\mathfrak{B}),$  $r \in
	\mathbb{R}$.
\end{enumerate}
For any $ \Phi \subset C(I,X)$, we define
\begin{equation*}
	\Phi(s)=\big\{\omega(s)\in X: \omega \in \Phi \big\} \; \mbox{and}\; \int_{0}^{t}\Phi(s)ds=\bigg\{ \int_{0}^{t}\omega(s)ds: \omega \in
	\Phi\bigg\},\mbox{for}\; t \in I.
\end{equation*}
\begin{prop}\label{prop2.7}\cite{GL}
	If $\Phi \subset C(I,X)$ is equicontinuous and bounded, then $t\rightarrow
	\psi(\Phi(t))$ is continuous on $I$, and
	\begin{equation*}
		\psi\bigg(\int_{0}^{t}\Phi(s)ds\bigg) \leq
		\int_{0}^{t}\psi(\Phi(s))ds \; \mbox{for} \; t\in I , \; \psi(\Phi)=\max_{s\in I}\psi(\Phi(s)).
	\end{equation*}
\end{prop}
\begin{prop}\label{prop2.8}\cite{HM}
Let $\{\omega_{n}:I \rightarrow X\}$ be a sequence of Bochner integrable functions  with  $\left\| \omega_{n}(t)\right\|\leq m(t)$ a.e. $t\in I$ and every $n\geq 1$, where $m\in
	L^{1}(I,\mathbb{R}^{+})$. Then  the function
	$\phi(t)=\psi(\{\omega_{n}(t)\}_{n=1}^{\infty})$ belongs to
	$L^{1}(I,\mathbb{R}^{+})$ and 
	\begin{equation*}
		\psi\bigg(\bigg\{\int_{0}^{t}\omega_{n}(\tau)d\tau:n\geq 1\bigg\}\bigg) \leq 2
		\int_{0}^{t}\phi(\tau)d\tau.
	\end{equation*}
\end{prop}
\begin{prop}\label{prop2.9}\cite{DB}
	If $\Phi$ is bounded, then for $\epsilon>0$, there is a sequence
	$\{w_{n}\}_{n=1}^{\infty}\subset \Phi$ satisfying
	\begin{equation*}
		\psi(\Phi)\leq 2\psi(\{w_{n}\}_{n=1}^{\infty})+\epsilon.
	\end{equation*}
\end{prop}
\section{Linear Problem}
We consider a real valued linear FDE
\begin{equation}\label{eq3.1}
D_{0}^{\gamma,\delta}\omega(t)=c\omega(t)+\eta(t),\; t>0
\end{equation}
\begin{equation}\label{eq3.2}
(g_{(1-\delta)(2-\gamma)}*\omega)(0)=x, (g_{(1-\delta)(2-\gamma)}*\omega)^{'}(0)=y,
\end{equation}
where $c$ is a real constant, $\eta:[0,T]\rightarrow \mathbb{R}$ and $g_{\gamma}(t)=\frac{t^{\gamma}-1}{\Gamma(\gamma)}, \gamma >0$.\\
The next two Lemmas give the insight about the structure of solution.
\begin{lemma}\label{lemma1}
If $f\in C^{1}([0,T],\mathbb{R})$, then $$\omega(t)=t^{\gamma+\delta(2-\gamma)-2}E_{\gamma,\gamma+\delta(2-\gamma)-1}(ct^{\gamma})x+t^{\gamma+\delta(2-\gamma)-1}E_{\gamma,\gamma+\delta(2-\gamma)}(ct^{\gamma})y+\int_{0}^{t}(t-s)^{\gamma-1}E_{\gamma,\gamma}(c(t-s)^{\gamma})\eta(s)ds$$ is  solution of the problem \eqref{eq3.1}-\eqref{eq3.2}.
\end{lemma}
\begin{proof}
	We define\\ $\omega_{1}(t)=t^{\gamma+\delta(2-\gamma)-2}E_{\gamma,\gamma+\delta(2-\gamma)-1}(ct^{\gamma})x+t^{\gamma+\delta(2-\gamma)-1}E_{\gamma,\gamma+\delta(2-\gamma)}(ct^{\gamma})y$ and\\ $\omega_{2}(t)=\int_{0}^{t}(t-s)^{\gamma-1}E_{\gamma,\gamma}(c(t-s)^{\gamma})\eta(s)ds.$\\ Then
\begin{eqnarray*}
&&J_{t}^{(1-\delta)(2-\gamma)}\omega_{2}(t)\\
&&=\frac{1}{\Gamma((1-\delta)(2-\gamma))}\int_{0}^{t}(t-\tau)^{(1-\delta)(2-\gamma)-1}\int_{0}^{\tau}(\tau-s)^{\gamma-1}E_{\gamma,\gamma}(c(\tau-s)^{\gamma})\eta(s)ds d\tau\\	
	&&= \frac{1}{\Gamma((1-\delta)(2-\gamma))}\int_{0}^{t}\eta(s)\int_{s}^{t}(t-\tau)^{(1-\delta)(2-\gamma)-1}(\tau-s)^{\gamma-1}E_{\gamma,\gamma}(c(\tau-s)^{\gamma})d\tau ds\\
	&&= \int_{0}^{t}(t-s)^{1-\delta(2-\gamma)}E_{\gamma,2-\delta(2-\gamma)}(c(t-s)^{\gamma})\eta(s)ds.
\end{eqnarray*}
Hence
	\begin{eqnarray}\label{eq3.3}
	&&\frac{d}{dt}J_{t}^{(1-\delta)(2-\gamma)}\omega_{2}(t)=\frac{d}{dt}\int_{0}^{t}(t-s)^{1-\delta(2-\gamma)}E_{\gamma,2-\delta(2-\gamma)}(c(t-s)^{\gamma})\eta(s)ds \nonumber\\&&=\int_{0}^{t}(t-s)^{-\delta(2-\gamma)}E_{\gamma,1-\delta(2-\gamma)}(c(t-s)^{\gamma})\eta(s)ds \nonumber\\&&=\int_{0}^{t}u^{-\delta(2-\gamma)}E_{\gamma,1-\delta(2-\gamma)}(cu^{\gamma})\eta(t-u)du.
	\end{eqnarray}
Now using \eqref{eq3.3}, we have
\begin{eqnarray*}
&&\frac{d^{2}}{dt^{2}}J_{t}^{(1-\delta)(2-\gamma)}\omega_{2}(t)=\frac{d}{dt}\int_{0}^{t}u^{-\delta(2-\gamma)}E_{\gamma,1-\delta(2-\gamma)}(cu^{\gamma})\eta(t-u)du\\ 
&&=\int_{0}^{t}u^{-\delta(2-\gamma)}E_{\gamma,1-\delta(2-\gamma)}(cu^{\gamma})f^{\prime}(t-u)du+t^{-\delta(2-\gamma)}E_{\gamma,1-\delta(2-\gamma)}(ct^{\gamma})\eta(0)\\
&&=\int_{0}^{t}(t-u)^{-\delta(2-\gamma)}E_{\gamma,1-\delta(2-\gamma)}(c(t-u)^{\gamma})f^{\prime}(\omega)du+t^{-\delta(2-\gamma)}E_{\gamma,1-\delta(2-\gamma)}(ct^{\gamma})\eta(0).
\end{eqnarray*}
Then as per definition
\begin{eqnarray*}
&&D_{0}^{\gamma,\delta}\omega_{2}(t)=J_{t}^{\delta(2-\gamma)}\frac{d^{2}}{dt^{2}}J_{t}^{(1-\delta)(2-\gamma)}\omega_{2}(t)\\
&&=\eta(t)-E_{\gamma}(ct^{\gamma})\eta(0)+c\int_{0}^{t}(t-u)^{\gamma-1}E_{\gamma,\gamma}(c(t-u)^{\gamma})\eta(\omega)du+E_{\gamma}(ct^{\gamma})\eta(0)\\
&&=\eta(t)+c\int_{0}^{t}(t-u)^{\gamma-1}E_{\gamma,\gamma}(c(t-u)^{\gamma})\eta(\omega)du.
\end{eqnarray*}
Hence,
\begin{eqnarray*}
D_{0}^{\gamma,\delta}\omega(t)&=& D_{0}^{\gamma,\delta}\omega_{1}(t)+D_{0}^{\gamma,\delta}\omega_{2}(t)\\&=&c(t^{\gamma+\delta(2-\gamma)-2}E_{\gamma,\gamma+\delta(2-\gamma)-1}(ct^{\gamma})x+t^{\gamma+\delta(2-\gamma)-1}E_{\gamma,\gamma+\delta(2-\gamma)}(ct^{\gamma})y\\&+&c\int_{0}^{t}(t-u)^{\gamma-1}E_{\gamma,\gamma}(c(t-u)^{\gamma})\eta(\omega)du)+\eta(t)=c\omega(t)+\eta(t), t \in (0,T].
\end{eqnarray*}
Clearly,
\begin{eqnarray*}
&&J_{t}^{(1-\delta)(2-\gamma)}\omega(t)\\
&&=E_{\gamma,1}(ct^{\gamma})x+t E_{\gamma,2}(c t^{\gamma})y+\int_{0}^{t}(t-s)^{1-\delta(2-\gamma)}E_{\gamma,2-\delta(2-\gamma)}(c(t-s)^{\gamma})\eta(s)ds \rightarrow x
\end{eqnarray*}
and
\begin{eqnarray*}
&&\frac{d}{dt}J_{t}^{(1-\delta)(2-\gamma)}\omega(t)\\
&&=ct^{\gamma-1}E_{\gamma,\gamma}(ct^{\gamma})x+E_{\gamma,1}(ct^{\gamma})y+\int_{0}^{t}(t-s)^{-\delta(2-\gamma)}E_{\gamma,1-\delta(2-\gamma)}(c(t-s)^{\gamma})\eta(s)ds \rightarrow y
\end{eqnarray*}

as  $t \rightarrow 0,$ by dominated convergence theorem.\\
Hence $\omega(t)=t^{\gamma+\delta(2-\gamma)-2}E_{\gamma,\gamma+\delta(2-\gamma)-1}(ct^{\gamma})x+t^{\gamma+\delta(2-\gamma)-1}E_{\gamma,\gamma+\delta(2-\gamma)}(ct^{\gamma})y+\int_{0}^{t}(t-s)^{\gamma-1}E_{\gamma,\gamma}(c(t-s)^{\gamma})\eta(s)ds$ is the solution of the problem \eqref{eq3.1}-\eqref{eq3.2}. 
\end{proof}
\begin{lemma}
If $\eta\in J^{\delta(2-\gamma)}(L^{1})$, then 	$$\omega(t)=t^{\gamma+\delta(2-\gamma)-2}E_{\gamma,\gamma+\delta(2-\gamma)-1}(ct^{\gamma})x+t^{\gamma+\delta(2-\gamma)-1}E_{\gamma,\gamma+\delta(2-\gamma)}(ct^{\gamma})y+\int_{0}^{t}(t-s)^{\gamma-1}E_{\gamma,\gamma}(c(t-s)^{\gamma})\eta(s)ds$$ is the solution of  \eqref{eq3.1}-\eqref{eq3.2}.
\end{lemma}
\begin{proof}
Similarly as in Lemma \ref{lemma1}, define $$\omega_{1}(t)=t^{\gamma+\delta(2-\gamma)-2}E_{\gamma,\gamma+\delta(2-\gamma)-1}(ct^{\gamma})x+t^{\gamma+\delta(2-\gamma)-1}E_{\gamma,\gamma+\delta(2-\gamma)}(ct^{\gamma})y$$ and $$\omega_{2}(t)=\int_{0}^{t}(t-s)^{\gamma-1}E_{\gamma,\gamma}(c(t-s)^{\gamma})\eta(s)ds.$$ Then using \eqref{eq3.3}, we have
\begin{eqnarray*}
D_{t}^{\gamma,\delta}\omega_{2}(t)&=&J_{t}^{\delta(2-\gamma)}\frac{d}{dt}\int_{0}^{t}(t-s)^{-\delta(2-\gamma)}E_{\gamma,1-\delta(2-\gamma)}(c(t-s)^{\gamma})\eta(s)ds\\&=&J_{t}^{\delta(2-\gamma)}\frac{d}{dt}I.
\end{eqnarray*}
Using Theorem $5.1$ of \cite{haubold2011mittag}, we have
\begin{equation*}
I=c\int_{0}^{t}(t-s)^{\gamma-\delta(2-\gamma)}E_{\gamma,\gamma-\delta(2-\gamma)+1}(c(t-s)^{\gamma})\eta(s)ds+J_{t}^{1-\delta(2-\gamma)}\eta(t).
\end{equation*}
Since $\eta \in J^{\delta(2-\gamma)}(L^{1})$, there exist a function $\phi \in L^{1}(0,T)$ such that $\eta(t)=J^{\delta(2-\gamma)}\phi(t)$. Now
\begin{eqnarray}\label{eq3.4}
&&\frac{dI}{dt}=c\int_{0}^{t}(t-s)^{\gamma-\delta(2-\gamma)-1}E_{\gamma,\gamma-\delta(2-\gamma)}(c(t-s)^{\gamma})\eta(s)ds+\frac{d}{dt}J_{t}^{1-\delta(2-\gamma)}\eta(t)\nonumber\\
&&=c\int_{0}^{t}(t-s)^{\gamma-\delta(2-\gamma)-1}E_{\gamma,\gamma-\delta(2-\gamma)}(c(t-s)^{\gamma})\eta(s)ds+\frac{d}{dt}J_{t}^{1-\delta(2-\gamma)}J_{t}^{\delta(2-\gamma)}\phi(t)\nonumber\\
&&=c\int_{0}^{t}(t-s)^{\gamma-\delta(2-\gamma)-1}E_{\gamma,\gamma-\delta(2-\gamma)}(c(t-s)^{\gamma})\eta(s)ds+\phi(t).
\end{eqnarray}
Using \eqref{eq3.4}, we have
\begin{eqnarray*}
D_{0}^{\gamma,\delta}\omega_{2}(t)&=&c\int_{0}^{t}(t-s)^{\gamma-1}E_{\gamma,\gamma}(c(t-s)^{\gamma})\eta(s)ds+J_{t}^{\delta(2-\gamma)}\phi(t)\\
&=& c\int_{0}^{t}(t-s)^{\gamma-1}E_{\gamma,\gamma}(c(t-s)^{\gamma})\eta(s)ds+\eta(t).
\end{eqnarray*}
Hence
\begin{equation*}
D_{0}^{\gamma,\delta}\omega(t)=\eta(t)+c \omega(t).
\end{equation*}
\end{proof}
\section{Linear Abstract Hilfer fractional Cauchy problem}
Now, in this section we discuss the following type of Cauchy problem
\begin{equation}\label{eq4.1}
D_{0}^{\gamma,\delta}\omega(t)=\mathcal{A}\omega(t)
\end{equation}
\begin{equation}\label{eq4.2}
(g_{(1-\delta)(2-\gamma)}*\omega)(0)=x, (g_{(1-\delta)(2-\gamma)}*\omega)^{'}(0)=0
\end{equation}
in a Banach space $X$, where $\mathcal{A}:D(\mathcal{A}) \subset X \rightarrow X$ is a linear densely defined operator on $X$.\\

Next, we generalize the properties of  $ \{t^{\gamma+\delta(2-\gamma)-2}E_{\gamma,\gamma+\delta(2-\gamma)-1}(ct^{\gamma})\}_{t>0} $ to introduce a family of bounded linear operators.
\begin{definition} We define general fractional cosine operator functions $\{C_{\gamma,\delta}(t)\}_{t>0}$ of order $\gamma \in (1,2)$ and type $0 \leq \delta\leq 1\; $ as a family of bounded linear operators such that
\begin{enumerate}[label=(\roman*)]
\item $C_{\gamma,\delta}(t)$ is strongly continuous;
\item $\lim_{t \rightarrow 0+} \frac{C_{\gamma,\delta}(t)x}{t^{\gamma+\delta(2-\gamma)-2}}=\frac{x}{\Gamma(\gamma+\delta(2-\gamma)-1)}$;
\item $C_{\gamma,\delta}(t)C_{\gamma,\delta}(s)=C_{\gamma,\delta}(s)C_{\gamma,\delta}(t)$ for all $t,s >0$.
\item
 \begin{eqnarray}\label{eq4.5}
&& C_{\gamma,\delta}(s)J_{t}^{\gamma}C_{\gamma,\delta}(t)-J_{s}^{\gamma}C_{\gamma,\delta}(s)C_{\gamma,\delta}(t)=\frac{s^{\gamma+\delta(2-\gamma)-2}}{\Gamma(\gamma+\delta(2-\gamma)-1)}J_{t}^{\gamma}C_{\gamma,\delta}(t)- \nonumber \\ && \frac{t^{\gamma+\delta(2-\gamma)-2}}{\Gamma(\gamma+\delta(2-\gamma)-1)}J_{s}^{\gamma}C_{\gamma,\delta}(s), \; \; t,s >0.
\end{eqnarray}
\end{enumerate}
\end{definition}


\begin{definition}
The generator of a general fractional cosine operator function of order $1<\gamma<2$ and type $\delta \in [0,1]$ on Banach space $X$ as a linear operator with domain 
$D(\mathcal{A})=\bigg\{x\in X:\lim_{t\rightarrow 0+}\bigg(\frac{C_{\gamma,\delta}(t)x-\frac{t^{\gamma+\delta(2-\gamma)-2}}{\Gamma(\gamma+\delta(2-\gamma)-1)}x}{t^{2\gamma+\delta(2-\gamma)-2}}\bigg) \;\mbox{exists}\bigg\}$\
and defined as
$\mathcal{A}:D(\mathcal{A}) \subset X \rightarrow X$ with
$$\mathcal{A}x=\Gamma(2\gamma+\delta(2-\gamma)-1)\lim_{t\rightarrow 0+}\bigg(\frac{C_{\gamma,\delta}(t)x-\frac{t^{\gamma+\delta(2-\gamma)-2}}{\Gamma(\gamma+\delta(2-\gamma)-1)}x}{t^{2\gamma+\delta(2-\gamma)-2}}\bigg).$$

\end{definition}

\begin{prop}\label{prop1}
Let $\{C_{\gamma,\delta}(t)\}_{t>0}$ be a general fractional cosine operator function of order $1<\gamma<2$ and type $\delta \in [0,1]$ and  $\mathcal{A}$ being its generator, then
\begin{enumerate}[label=\alph*)]

\item $C_{\gamma,\delta}(t)$ maps $D(\mathcal{A})$ to  $D(\mathcal{A})$  and $C_{\gamma,\delta}(t)\mathcal{A}x=\mathcal{A} C_{\gamma,\delta}(t)x, \; \; x \in D(\mathcal{A});$
\item $\forall x \in X, t>0 \; J_{t}^{\gamma}C_{\gamma,\delta}(t)x \in D(\mathcal{A}) $ and
$$ C_{\gamma,\delta}(t)x=\frac{t^{\gamma+\delta(2-\gamma)-2}}{\Gamma(\gamma+\delta(2-\gamma)-1)}x+\mathcal{A}J_{t}^{\gamma}C_{\gamma,\delta}(t)x;$$
\item For $ x \in D(\mathcal{A})$
$$ C_{\gamma,\delta}(t)x=\frac{t^{\gamma+\delta(2-\gamma)-2}}{\Gamma(\gamma+\delta(2-\gamma)-1)}x+J_{t}^{\gamma}C_{\gamma,\delta}(t)\mathcal{A}x;$$

\item $\mathcal{A}$ is a closed and densely defined operator.
\item $\mathcal{A}$ corresponds to at most one general fractional cosine operator function of order $1<\gamma<2$ and type $\delta \in [0,1]$.
\item $C_{\gamma,\delta}(\cdot)x\in C^{1}((0,\infty);X)$ if $x\in D(\mathcal{A})$.
\end{enumerate}
\end{prop}
\begin{proof}
\begin{enumerate}[label=\alph*)]
\item

For $x \in D(\mathcal{A})$ and all $t,s >0$, we have
$$\frac{C_{\gamma,\delta}(s)C_{\gamma,\delta}(t)x-\frac{s^{\gamma+\delta(2-\gamma)-2}}{\Gamma(\gamma+\delta(2-\gamma)-1)}C_{\gamma,\delta}(t)x}{t^{2\gamma+\delta(2-\gamma)-2}}=\frac{C_{\gamma,\delta}(t)\bigg(C_{\gamma,\delta}(s)x-\frac{s^{\gamma+\delta(2-\gamma)-2}}{\Gamma(\gamma+\delta(2-\gamma)-1)}x\bigg)}{t^{2\gamma+\delta(2-\gamma)-2}}.$$
Since $C_{\gamma,\delta}(t)$ is bounded linear operator, we have
$$\Gamma(2\gamma+\delta(2-\gamma)-1)\lim_{s\rightarrow 0+}\frac{C_{\gamma,\delta}(s)C_{\gamma,\delta}(t)x-\frac{s^{\gamma+\delta(2-\gamma)-2}}{\Gamma(\gamma+\delta(2-\gamma)-1)}C_{\gamma,\delta}(t)x}{t^{2\gamma+\delta(2-\gamma)-2}}=C_{\gamma,\delta}(t)\mathcal{A}x.$$
Hence $C_{\gamma,\delta}(t)x \in D(\mathcal{A})$ and $\mathcal{A}C_{\gamma,\delta}(t)x=C_{\gamma,\delta}(t)\mathcal{A}x.$
\item
Let $x\in X$,
\begin{eqnarray}
&&\Gamma(2\gamma+\delta(2-\gamma)-1)\lim_{s\rightarrow 0+}\frac{\bigg(C_{\gamma,\delta}(s)-\frac{s^{\gamma+\delta(2-\gamma)-2}}{\Gamma(\gamma+\delta(2-\gamma)-1)}\bigg)J_{t}^{\gamma}C_{\gamma,\delta}(t)x}{s^{2\gamma+\delta(2-\gamma)-2}}\nonumber\\&&=\Gamma(2\gamma+\delta(2-\gamma)-1)\lim_{s\rightarrow 0+}\frac{J_{s}^{\gamma}C_{\gamma,\delta}(s)\bigg(C_{\gamma,\delta}(t)x-\frac{t^{\gamma+\delta(2-\gamma)-2}x}{\Gamma(\gamma+\delta(2-\gamma)-1)}\bigg)}{s^{2\gamma+\delta(2-\gamma)-2}}
\end{eqnarray}
Now we claim that $\lim_{s \rightarrow 0+}\Gamma(2\gamma+\delta(2-\gamma)-1)\lim_{s\rightarrow 0+}\frac{J_{s}^{\gamma}C_{\gamma,\delta}(s)x}{s^{2\gamma+\delta(2-\gamma)-2}}=x.$
\begin{align*}
&\bigg\lVert\Gamma(2\gamma+\delta(2-\gamma)-1)\frac{J_{s}^{\gamma}C_{\gamma,\delta}(s)x}{s^{2\gamma+\delta(2-\gamma)-2}}- x \bigg\rVert\\
&=\bigg\lVert\frac{\Gamma(2\gamma+\delta(2-\gamma)-1)}{\Gamma(\gamma)}\int_{0}^{s}s^{2-2\gamma-\delta(2-\gamma)}(s-\tau)^{\gamma-1}C_{\gamma,\delta}(\tau)x d\tau - x \bigg\rVert\\
&=\bigg\lVert\frac{\Gamma(2\gamma+\delta(2-\gamma)-1)}{\Gamma(\gamma)}\int_{0}^{1}s^{2-\gamma-\delta(2-\gamma)}(1-\tau)^{\gamma-1}C_{\gamma,\delta}(s\tau)x d\tau - x \bigg\rVert\\
&=\bigg\lVert\frac{\Gamma(2\gamma+\delta(2-\gamma)-1)}{\Gamma(\gamma)\Gamma(\gamma+\delta(2-\gamma)-1)}\int_{0}^{1}(1-\tau)^{\gamma-1}\tau^{\gamma+\delta(2-\gamma)-2}\Gamma(\gamma+\delta(2-\gamma)-1)\frac{C_{\gamma,\delta}(s\tau)}{(s\tau)^{\gamma+\delta(2-\gamma)-2}}x d\tau - x \bigg\rVert\\
&=\bigg\lVert\frac{\Gamma(2\gamma+\delta(2-\gamma)-1)}{\Gamma(\gamma)\Gamma(\gamma+\delta(2-\gamma)-1)}\int_{0}^{1}(1-\tau)^{\gamma-1}\tau^{\gamma+\delta(2-\gamma)-2}\Gamma(\gamma+\delta(2-\gamma)-1)\frac{C_{\gamma,\delta}(s\tau)}{(s\tau)^{\gamma+\delta(2-\gamma)-2}}x d\tau\\ & - \frac{\Gamma(2\gamma+\delta(2-\gamma)-1)}{\Gamma(\gamma)\Gamma(\gamma+\delta(2-\gamma)-1)}\int_{0}^{1}(1-\tau)^{\gamma-1}\tau^{\gamma+\delta(2-\gamma)-2}d\tau\bigg\rVert\\
& \leq \frac{\Gamma(2\gamma+\delta(2-\gamma)-1)}{\Gamma(\gamma)\Gamma(\gamma+\delta(2-\gamma)-1)}\int_{0}^{1}(1-\tau)^{\gamma-1}\tau^{\gamma+\delta(2-\gamma)-2}d\tau \sup_{\tau \in [0,1]}\bigg\lVert \Gamma(\gamma+\delta(2-\gamma)-1)\frac{C_{\gamma,\delta}(s\tau)}{(s\tau)^{\gamma+\delta(2-\gamma)-2}}x-x \bigg\rVert.
\end{align*}
This implies
\begin{equation}\label{eqa}
\lim_{s \rightarrow 0+}\Gamma(2\gamma+\delta(2-\gamma)-1)\lim_{s\rightarrow 0+}\frac{J_{s}^{\gamma}C_{\gamma,\delta}(s)x}{s^{2\gamma+\delta(2-\gamma)-2}}=x.
\end{equation}
\item
Let $x\in D(\mathcal{A})$
\begin{align*}
&C_{\gamma,\delta}(t)x-\frac{t^{\gamma+\delta(2-\gamma)-2}}{\Gamma(\gamma+\delta(2-\gamma)-1)}x=\mathcal{A}J_{t}^{\gamma}C_{\gamma,\delta}(t)x\\
&=\frac{\Gamma(2\gamma+\delta(2-\gamma)-1)}{\Gamma(\gamma)}\lim_{s \rightarrow 0+}\bigg(\frac{C_{\gamma,\delta}(s)-\frac{s^{\gamma+\delta(2-\gamma)-2}}{\Gamma(\gamma+\delta(2-\gamma)-1)}}{s^{2\gamma+\delta(2-\gamma)-2}}\bigg)\int_{0}^{t}(t-\tau)^{\gamma-1}C_{\gamma,\delta}(\tau)x d\tau\\
&= \frac{\Gamma(2\gamma+\delta(2-\gamma)-1)}{\Gamma(\gamma)}\lim_{s \rightarrow 0+}\int_{0}^{t}(t-\tau)^{\gamma-1}C_{\gamma,\delta}(\tau)\bigg(\frac{C_{\gamma,\delta}(s)-\frac{s^{\gamma+\delta(2-\gamma)-2}}{\Gamma(\gamma+\delta(2-\gamma)-1)}}{s^{2\gamma+\delta(2-\gamma)-2}}\bigg)x d\tau\\
&= \frac{\Gamma(2\gamma+\delta(2-\gamma)-1)}{\Gamma(\gamma)}\int_{0}^{t}(t-\tau)^{\gamma-1}C_{\gamma,\delta}(\tau)\lim_{s \rightarrow 0+}\bigg(\frac{C_{\gamma,\delta}(s)-\frac{s^{\gamma+\delta(2-\gamma)-2}}{\Gamma(\gamma+\delta(2-\gamma)-1)}}{s^{2\gamma+\delta(2-\gamma)-2}}\bigg)x d\tau\\
&= J_{t}^{\gamma}C_{\gamma,\delta}(t)\mathcal{A}x.
\end{align*}
\item Let $x_{n}\in D(\mathcal{A}), x_{n}\rightarrow x$ and $\mathcal{A}x_{n}\rightarrow y$ as $n\rightarrow \infty$.
\begin{eqnarray}
C_{\gamma,\delta}(t)x-\frac{t^{\gamma+\delta(2-\gamma)-2}}{\Gamma(\gamma+\delta(2-\gamma)-1)}x&=&\lim_{n\rightarrow \infty}\bigg(C_{\gamma,\delta}(t)x_{n}-\frac{t^{\gamma+\delta(2-\gamma)-2}}{\Gamma(\gamma+\delta(2-\gamma)-1)}x_{n}\bigg)\nonumber\\
&=&\lim_{n\rightarrow \infty}J_{t}^{\gamma}C_{\gamma,\delta}(t)\mathcal{A}x_{n}\nonumber\\
&=&\lim_{n\rightarrow \infty}\frac{1}{\Gamma(\gamma)}\int_{0}^{t}(t-s)^{\gamma-1}C_{\gamma,\delta}(s)\mathcal{A}x_{n}ds \nonumber\\
&=&\frac{1}{\Gamma(\gamma)}\int_{0}^{t}(t-s)^{\gamma-1}C_{\gamma,\delta}(s)yds \nonumber\\
&=&J_{t}^{\gamma}C_{\gamma,\delta}(t)y.
\end{eqnarray}
Using definition of $D(\mathcal{A})$and \eqref{eqa}, we have
\begin{eqnarray}
\mathcal{A}x&=&\Gamma(2\gamma+\delta(2-\gamma)-1)\lim_{t\rightarrow 0+}\bigg(\frac{C_{\gamma,\delta}(t)x-\frac{t^{\gamma+\delta(2-\gamma)-2}}{\Gamma(\gamma+\delta(2-\gamma)-1)}x}{t^{2\gamma+\delta(2-\gamma)-2}}\bigg)\nonumber\\
&=&\Gamma(2\gamma+\delta(2-\gamma)-1)\lim_{t\rightarrow 0+}\frac{J_{t}^{\gamma}C_{\gamma,\delta}(t)y}{t^{2\gamma+\delta(2-\gamma)-2}}=y.
\end{eqnarray}
Hence $\mathcal{A}$ is closed operator.\\
For $x\in X$, set $x_{t}=J_{t}^{\gamma}C_{\gamma,\delta}(t)x$. Then $x_{t}\in D(\mathcal{A})$ from $(b)$ and by \eqref{eqa} we have $\Gamma(2\gamma+\delta(2-\gamma)-1)\lim_{t\rightarrow 0+}\frac{J_{t}^{\gamma}C_{\gamma,\delta}(t)x}{t^{2\gamma+\delta(2-\gamma)-2}}=x$. Thus $\overline{D(\mathcal{A})}=X$.
\item Follows from $(c),(d)$ and Titchmarsh's Theorem.
\item
Next claim $C_{\gamma,\delta}(t)x \in C^{1}((0,\infty);X)$ for any $x\in D(\mathcal{A})$.
\begin{eqnarray*}
	&&\frac{d}{dt}C_{\gamma,\delta}(t)x=\frac{d}{dt}\bigg(\frac{t^{\gamma+\delta(2-\gamma)-2}}{\Gamma(\gamma+\delta(2-\gamma)-1)}x+J_{t}^{\gamma}C_{\gamma,\delta}(t)\mathcal{A}x\bigg)\\
	&&=\frac{(\gamma+\delta(2-\gamma)-2)}{\Gamma(\gamma+\delta(2-\gamma)-1)}t^{\gamma+\delta(2-\gamma)-3}x+J_{t}^{\gamma-1}C_{\gamma,\delta}(t)\mathcal{A}x\\
	&&=\frac{(\gamma+\delta(2-\gamma)-2)}{\Gamma(\gamma+\delta(2-\gamma)-1)}t^{\gamma+\delta(2-\gamma)-3}x+\frac{1}{\Gamma(\gamma-1)}\int_{0}^{t}(t-\tau)^{\gamma-2}C_{\gamma,\delta}(\tau)\mathcal{A}x d\tau\\
	&&=\frac{(\gamma+\delta(2-\gamma)-2)}{\Gamma(\gamma+\delta(2-\gamma)-1)}t^{\gamma+\delta(2-\gamma)-3}x+\frac{t^{\gamma-1}}{\Gamma(\gamma-1)}\int_{0}^{1}(1-\tau)^{\gamma-2}C_{\gamma,\delta}(t\tau)\mathcal{A}x d\tau.
\end{eqnarray*}
The dominated convergence theorem gives $C_{\gamma,\delta}(.)x \in C^{1}((0,\infty);X)$.
\end{enumerate}
\end{proof}
\begin{definition}
We define a family of bounded linear operators $\{\mathfrak{T}(t)\}_{t>0}$ as a solution operator  of  \ref{eq4.1}-\ref{eq4.2} if
\begin{enumerate}
\item  $\mathfrak{T}(\cdot)x \in C((0,\infty),X)$ and
$$\lim_{t \rightarrow 0+} \Gamma(\gamma+\delta(2-\gamma)-1)\frac{\mathfrak{T}(t)}{t^{\gamma+\delta(2-\gamma)-2}}x=x \; \mbox{for} \; x\in X;$$
\item $\mathfrak{T}(t)D(\mathcal{A}) \subset D(\mathcal{A}),\; \mathcal{A}\mathfrak{T}(t)x=\mathfrak{T}(t)\mathcal{A}x \; \mbox{for}\; x \in D(\mathcal{A});$
\item For all $ x \in D(\mathcal{A})$
$$ \mathfrak{T}(t)x=\frac{t^{\gamma+\delta(2-\gamma)-2}}{\Gamma(\gamma+\delta(2-\gamma)-1)}x+J_{t}^{\gamma}\mathfrak{T}(t)\mathcal{A}x.$$
\end{enumerate}

\end{definition}
\begin{definition}
$\omega \in C((0,\infty),X)$ is a mild solution if $J_{t}^{\gamma}\omega(t) \in D(\mathcal{A})$ and $ \omega(t)=\frac{t^{\gamma+\delta(2-\gamma)-2}}{\Gamma(\gamma+\delta(2-\gamma)-1)}+\mathcal{A}J_{t}^{\gamma}\omega(t),\; t\in (0,\infty).$
\end{definition}
\begin{definition}
$\omega \in C((0,\infty),X)$ is  a strong solution if $\omega(t) \in D(\mathcal{A})$ for  $t>0$ and $ D_{0}^{\gamma,\delta}\omega(t)$ is continuous on $(0,\infty)$ and \ref{eq4.1}-\ref{eq4.2} holds.
\end{definition}
\begin{thm} 
Let $\mathcal{A}$ be the generator of a general fractional cosine operator function of order $1<\gamma<2$ and type $\delta \in [0,1$ i.e. $C_{\gamma,\delta}(t)$, then $C_{\gamma,\delta}(t)x$ is strong solution for $x\in D(\mathcal{A})$.
\end{thm}
\begin{proof}
Since
$$ C_{\gamma,\delta}(t)x=\frac{t^{\gamma+\delta(2-\gamma)-2}}{\Gamma(\gamma+\delta(2-\gamma)-1)}x+J_{t}^{\gamma}C_{\gamma,\delta}(t)\mathcal{A}x,$$
we have
$$J_{t}^{(1-\delta)(2-\gamma)}C_{\gamma,\delta}(t)x=x+J_{t}^{2-\delta(2-\gamma)}C_{\gamma,\delta}(t)\mathcal{A}x.$$
Hence
\begin{eqnarray*}
&&\lim_{t \rightarrow 0+}J_{t}^{(1-\delta)(2-\gamma)}C_{\gamma,\delta}(t)x=x+\lim_{t \rightarrow 0+}J_{t}^{2-\delta(2-\gamma)}C_{\gamma,\delta}(t)\mathcal{A}x\\
&&=x+\lim_{t \rightarrow 0+}\frac{1}{\Gamma(2-\delta(2-\gamma))}\int_{0}^{t}(t-s)^{1-\delta(2-\gamma)}C_{\gamma,\delta}(s)\mathcal{A}x ds\\
&&=x+\lim_{t \rightarrow 0+}\frac{1}{\Gamma(2-\delta(2-\gamma))}\int_{0}^{1}t^{2-\delta(2-\gamma)}(1-u)^{1-\delta(2-\gamma)}C_{\gamma,\delta}(tu)\mathcal{A}x du\\
&&=x+\lim_{t \rightarrow 0+}\frac{1}{\Gamma(2-\delta(2-\gamma))}t^{\gamma}\int_{0}^{1}u^{\gamma+\delta(2-\gamma)-2}(1-u)^{1-\delta(2-\gamma)}(tu)^{2-\gamma-\delta(2-\gamma)}C_{\gamma,\delta}(tu)\mathcal{A}x du.\\
&& =x,
\end{eqnarray*}
by  Lebesgue's dominated convergence theorem. Similarly
\begin{eqnarray*}
&&\frac{d}{dt}J_{t}^{(1-\delta)(2-\gamma)}C_{\gamma,\delta}(t)x=J_{t}^{1-\delta(2-\gamma)}C_{\gamma,\delta}(t)\mathcal{A}x\\
&&=\frac{1}{\Gamma(1-\delta(2-\gamma))}\int_{0}^{t}(t-s)^{-\delta(2-\gamma)}C_{\gamma,\delta}(s)\mathcal{A}x ds\\
&&=\frac{1}{\Gamma(1-\delta(2-\gamma))}\int_{0}^{1}t^{1-\delta(2-\gamma)}(1-u)^{-\delta(2-\gamma)}C_{\gamma,\delta}(tu)\mathcal{A}x du\\
&&=\frac{1}{\Gamma(1-\delta(2-\gamma))}t^{\gamma-1}\int_{0}^{1}u^{\gamma+\delta(2-\gamma)-2}(1-u)^{-\delta(2-\gamma)}(tu)^{2-\gamma-\delta(2-\gamma)}C_{\gamma,\delta}(tu)\mathcal{A}x du\\
&& \rightarrow 0 \;\mbox{as}\; t \rightarrow 0+.
\end{eqnarray*}
Then
$$\frac{d^{2}}{dt^{2}}J_{t}^{(1-\delta)(2-\gamma)}C_{\gamma,\delta}(t)x=\frac{d}{dt}J_{t}^{1-\delta(2-\gamma)}C_{\gamma,\delta}(t)\mathcal{A}x=D_{0}^{\delta(2-\gamma)}C_{\gamma,\delta}(t)\mathcal{A}x.$$
Hence
$$D_{0}^{\gamma,\delta}C_{\gamma,\delta}(t)x=J_{t}^{\delta(2-\gamma)}D_{0}^{\delta(2-\gamma)}C_{\gamma,\delta}(t)\mathcal{A}x=C_{\gamma,\delta}(t)\mathcal{A}x=\mathcal{A}C_{\gamma,\delta}(t)x.$$

\end{proof}
\begin{lemma}
Let $\mathcal{A}$ generates a general fractional cosine operator function of order $1<\gamma<2$ and type $\delta \in [0,1]$. Then there exist a solution operator $S_{\gamma}(t)$ of the following FDE
\begin{equation}\label{eq4.8a}
D_{t}^{\gamma}[\omega(t)-x]=\mathcal{A}\omega(t), \; t>0,
\end{equation}
\begin{equation}\label{eq4.9a}
\omega(0)=x, w^{\prime}(0)=0.
\end{equation}
\end{lemma}
\begin{proof}
Define $S_{\gamma}(t)=J_{t}^{(1-\delta)(2-\gamma)}C_{\gamma,\delta}(t), t>0$. For $x \in X$,
\begin{eqnarray*}
&&\lim_{t\rightarrow 0+}S_{\gamma}(t)x=\lim_{t\rightarrow 0+}J_{t}^{(1-\delta)(2-\gamma)}C_{\gamma,\delta}(t)x\\
&&\frac{1}{\Gamma((1-\delta)(2-\gamma))}\lim_{t\rightarrow 0+}\int_{0}^{1}u^{\gamma+\delta(2-\gamma)-2}(1-u)^{(2-\gamma)(1-\delta)-1}(tu)^{2-\gamma-\delta(2-\gamma)}C_{\gamma,\delta}(tu)x du\\
&&=x,
\end{eqnarray*}
by dominated convergence theorem.\\
This gives, $S_{\gamma}(0)=I$. Then, $\{S_{\gamma}(t)\}_{t \geq 0}$  is strongly continuous. Now using property \ref{prop1},  we have
$$S_{\gamma}(t)x=J_{t}^{(1-\delta)(2-\gamma)}C_{\gamma,\delta}(t)x=x+J_{t}^{\gamma}J_{t}^{(1-\delta)(2-\gamma)}C_{\gamma,\delta}(t)\mathcal{A}x=x+J_{t}^{\gamma}S_{\gamma}(t)\mathcal{A}x,$$
for all $x \in  D(\mathcal{A})$,
Therefore $\{S_{\gamma}(t)\}_{t \geq 0}$ is a solution operator for the given  differential equation \eqref{eq4.8a}-\eqref{eq4.9a}.
\end{proof}

\begin{lemma}
If $\mathcal{A}$ generates a general fractional cosine operator  function of order $1<\gamma<2$ and type $\delta \in [0,1]$, then it generates general fractional sine function $\{S_{\gamma,\delta}(t)\}_{t \geq 0}$(see \cite{ZDM}) of order $\gamma$ and type $\delta$ also.
\end{lemma}
\begin{proof}
Define $S_{\gamma,\delta}(t)=\int_{0}^{t}C_{\gamma,\delta}(s)ds,\; t>0$, then for $x \in X$
\begin{equation}\label{eq4.6}
S_{\gamma,\delta}(t)x=\frac{t^{\gamma+\delta(2-\gamma)-1}}{\Gamma(\gamma+\delta(2-\gamma))}x+\mathcal{A}J_{t}^{\gamma}S_{\gamma,\delta}(t)x.
\end{equation}
Now 
\begin{eqnarray*}
&&\Gamma(\gamma+\delta(2-\gamma))\lim_{t\rightarrow 0+}\frac{S(t)y}{t^{\gamma+\delta(2-\gamma)-1}}=\Gamma(\gamma+\delta(2-\gamma))\lim_{t\rightarrow 0+}\frac{\int_{0}^{t}C_{\gamma,\delta}(s)y ds}{t^{\gamma+\delta(2-\gamma)-1}}\\
&&=\Gamma(\gamma+\delta(2-\gamma))\lim_{t\rightarrow 0+}t^{2-\gamma-\delta(2-\gamma)}\int_{0}^{1}C_{\gamma,\delta}(ts)y ds\\
&&=\Gamma(\gamma+\delta(2-\gamma))\lim_{t\rightarrow 0+}\int_{0}^{1}s^{\gamma+\delta(2-\gamma)-2}(ts)^{2-\gamma-\delta(2-\gamma)}C_{\gamma,\delta}(ts)y ds=y,
\end{eqnarray*}
by dominated convergence theorem.\\
The commutativity of $\{C_{\gamma,\delta}(t)\}_{t>0}$ implies the commutativity of $\{S_{\gamma,\delta}(t)\}_{t>0}$.\\
Using \eqref{eq4.6} and closedness of operator $\mathcal{A}$, we can get\\
$S_{\gamma,\delta}(s)J_{t}^{\gamma}S_{\gamma,\delta}(t)x-J_{s}^{\gamma}S_{\gamma,\delta}(s)S_{\gamma,\delta}(t)x=\frac{s^{\gamma+\delta(2-\gamma)-1}}{\Gamma(\gamma+\delta(2-\gamma))}J_{t}^{\gamma}S_{\gamma,\delta}(t)x-\frac{t^{\gamma+\delta(2-\gamma)-1}}{\Gamma(\gamma+\delta(2-\gamma))}J_{s}^{\gamma}S_{\gamma,\delta}(s)x;\; t, s>0.$\\
Therefore, $\{S_{\gamma,\delta}(t)\}_{t \geq 0}$ is a general fractional sine function.
\end{proof}

\begin{prop}
Define $R(\gamma)=\int_{0}^{\infty}e^{-\gamma t}C_{\gamma,\delta}(t)dt$. Let $R(\gamma)$ exist for some $\gamma=\gamma_{0}$. Then
$$ \gamma^{-\delta(2-\gamma)+1}R(\mu)-\mu^{-\delta(2-\gamma)+1}R(\gamma)=(\gamma^{\gamma}-\mu^{\gamma})R(\mu)R(\gamma) \; \mbox{for}\; \gamma,\mu \geq \gamma_{0}.$$
\end{prop}
\begin{proof}
Taking Laplace transform w.r.t $t$ and $s$ of L.H.S. of \eqref{eq4.5}, we have\\
$\int_{0}^{t}e^{-\gamma t}[C_{\gamma,\delta}(s)J_{t}^{\gamma}C_{\gamma,\delta}(t)-J_{s}^{\gamma}C_{\gamma,\delta}(s)C_{\gamma,\delta}(t)]dt=C_{\gamma,\delta}(s)\gamma^{-\gamma}R(\gamma)-J_{s}^{\gamma}C_{\gamma,\delta}(s)R(\gamma),$
\begin{equation}\label{eq4.11a}
\int_{0}^{t}e^{-\mu s}[C_{\gamma,\delta}(s)\gamma^{-\gamma}R(\gamma)-J_{s}^{\gamma}C_{\gamma,\delta}(s)R(\gamma)]ds=\gamma^{-\gamma}R(\mu)R(\gamma)-\mu^{-\gamma}R(\mu)R(\gamma).
\end{equation}
Taking Laplace transform of R.H.S  of \eqref{eq4.5} w.r.t $t$, we have
$$\gamma^{-\gamma}\frac{s^{\gamma+\delta(2-\gamma)-2}}{\Gamma(\gamma+\delta(2-\gamma)-1)}R(\gamma)-\gamma^{-\gamma-\delta(2-\gamma)+1}J_{s}^{\gamma}C_{\gamma,\delta}(s).$$
Again taking Laplace transform w.r.t  $s$, we have
\begin{equation}\label{eq4.12a}
 \gamma^{-\gamma}\mu^{-\gamma-\delta(2-\gamma)+1}R(\gamma)- \mu^{-\gamma}\gamma^{-\gamma-\delta(2-\gamma)+1}R(\mu).
\end{equation}
Equating the both sides \eqref{eq4.11a} and \eqref{eq4.12a}, we get the result.
\end{proof}
\begin{thm}
Let $\mathcal{A}$ generates a general fractional cosine operator function of order $1<\gamma<2$ and type $\delta \in [0,1]$, and Laplace transform of $C_{\gamma,\delta}(t)$ exists, then
$$R(\gamma^{\gamma},\mathcal{A})x=\gamma^{\delta(2-\gamma)-1}\int_{0}^{\infty}e^{-\gamma t}C_{\gamma,\delta}(t)xdt.$$
\end{thm}
\begin{proof}
For $x \in D(\mathcal{A})$, we have
\begin{equation}\label{l1}
C_{\gamma,\delta}(t)x=\frac{t^{\gamma+\delta(2-\gamma)-2}}{\Gamma(\gamma+\delta(2-\gamma)-1)}x+J_{t}^{\gamma}C_{\gamma,\delta}(t)\mathcal{A}x.
\end{equation}
Applying Laplace transform on both sides of \eqref{l1} , we get
$$R(\gamma)x=\gamma^{-\gamma-\delta(2-\gamma)+1}x+\gamma^{-\gamma}R(\gamma)\mathcal{A}x, \forall x \in D(\mathcal{A}).$$
Since $\overline{D(\mathcal{A})}=X$, we have
$$\gamma^{\delta(2-\gamma)-1}R(\gamma)x=\gamma^{-\gamma}x+\gamma^{\delta(2-\gamma)-1}\gamma^{-\gamma}R(\gamma)\mathcal{A}x \; \mbox{on} \; X$$
Hence $$\gamma^{\delta(2-\gamma)-1}R(\gamma)(\gamma^{\gamma}-\mathcal{A})x=x$$
and
$$(I-\gamma^{-\gamma}\mathcal{A})R(\gamma)x=\gamma^{-\gamma-\delta(2-\gamma)+1}x.$$
This gives
$$(\gamma^{\gamma}I-\mathcal{A})\gamma^{\delta(2-\gamma)-1}R(\gamma)x=x.$$
Thus $R(\gamma^{\gamma},\mathcal{A})x=\gamma^{\delta(2-\gamma)-1}\int_{0}^{\infty}e^{-\gamma t}C_{\gamma,\delta}(t)xdt.$
\end{proof}
\begin{remark}
A family  of bounded linear operators $\{C_{\gamma,\delta}(t)\}_{t>0}$ is a solution operator of \eqref{eq4.1}-\eqref{eq4.2} if and only if it is general fractional cosine operator function of order $1<\gamma<2$ and type $\delta \in [0,1]$.
\end{remark}

Consider the problem
\begin{equation}\label{eq4.7}
D_{0}^{\gamma,\delta}\omega(t)=\mathcal{A}\omega(t),\; t>0,
\end{equation}
\begin{equation}\label{eq4.8}
(g^{(1-\delta)(2-\gamma)}*\omega)(0)=x, (g^{(1-\delta)(2-\gamma)}*\omega)^{'}(0)=y,
\end{equation}
in a Banach space $X$ where $\mathcal{A}$ is closed linear operator on $X$.
\begin{definition}
$u \in C((0,\infty),X)$ is a mild solution of the problem \eqref{eq4.7}-\eqref{eq4.8} if $J_{t}^{\gamma}\omega(t) \in D(\mathcal{A})$ and
 $$\omega(t)=\frac{t^{\gamma+\delta(2-\gamma)-2}}{\Gamma(\gamma+\delta(2-\gamma)-1)}x+\frac{t^{\gamma+\delta(2-\gamma)-1}}{\Gamma(\gamma+\delta(2-\gamma))}y+\mathcal{A}J_{t}^{\gamma}\omega(t)\; t\in (0,\infty).$$
\end{definition}
\begin{thm}
Let $\mathcal{A}$ generates a general fractional cosine operator function of order $1<\gamma<2$ and type $\delta [0,1]$. Then for every $x, y\in X$, $C_{\gamma,\delta}(t)x+S_{\gamma,\delta}(t)y$ is the mild solution of  \eqref{eq4.7}-\eqref{eq4.8}.
\end{thm}
\section{Semilinear Abstract Hilfer Cauchy Problem}
We study the  following in-homogeneous problem
\begin{equation}\label{eq5.1}
D_{0}^{\gamma,\delta}\omega(t)=\mathcal{A}\omega(t)+\eta(t,\omega(t)), t \in (0,T],
\end{equation}
\begin{equation}\label{eq5.2}
(g_{(1-\delta)(2-\gamma)}*\omega)(0)=\omega_{1} \in X, (g_{(1-\delta)(2-\gamma)}*\omega)^{'}(0)=\omega_{2} \in X,
\end{equation}
where $\mathcal{A}$ generates a general fractional cosine operator function $C_{\gamma,\delta}(t)$  of order $1<\gamma<2$ and type $0 \leq \delta \leq 1$, such that Laplace transform of $C_{\gamma,\delta}(t)$ exist.\\

Applying Laplace transform on both sides of \eqref{eq5.1}, we get
$$\hat{\omega}(s)=s^{1-\delta(2-\gamma)}R(s^{\gamma},\mathcal{A})\omega_{1}+s^{-\delta(2-\gamma)}R(s^{\gamma},\mathcal{A})\omega_{2}+R(s^{\gamma},\mathcal{A})\hat{\eta}(s).$$
Now taking inverse Laplace transform , we have
$$\omega(t)=C_{\gamma,\delta}(t)\omega_{1}+S_{\gamma,\delta}(t)\omega_{2}+\int_{0}^{t}P_{\gamma,\delta}(t-s)\eta(s,\omega(s))ds,$$
where, $P_{\gamma,\delta}(t)=J_{t}^{1-\delta(2-\gamma)}C_{\gamma,\delta}(t)$.\\
Thus, we define the mild solution of the FDE \eqref{eq5.1}-\eqref{eq5.2} as $\omega \in C((0,T];X)$ such that $\omega$ satisfies
$$\omega(t)=C_{\gamma,\delta}(t)\omega_{1}+S_{\gamma,\delta}(t)\omega_{2}+\int_{0}^{t}P_{\gamma,\delta}(t-s)\eta(s,\omega(s))ds, t\in (0,T].$$
Let $I=[0,T]$ and $I'=(0,T]$, we define a space $Y$ as
$$Y=\{\omega \in C(I',X):\lim_{t \rightarrow 0+}t^{(2-\gamma-\delta(2-\gamma))}\omega(t)\; \mbox{exists and is finite} \}$$
and the norm on $Y$ is defined by  $\lVert \omega \rVert_{Y}=\sup_{t\in I'}\{t^{(1-\delta)(2-\gamma)} \lVert \omega(t)\rVert\}.$\\
Let $y:I\rightarrow X$ and $\omega(t)=t^{(2-\gamma)(\delta-1)}y(t), t\in I'$, then $\omega \in Y$ if and only if $y \in C(I,X)$ and $\lVert \omega \rVert_{Y}=\lVert y\rVert$.\\
Let $B_{r}(I)=\{y\in C(I,X):\lVert y\rVert\leq r\}$ and $B_{r}^{Y}(I')=\{\omega \in Y:\lVert \omega \rVert_{Y} \leq r\}.$
\begin{thm}
$\omega(t)=C_{\gamma,\delta}(t)x+S_{\gamma,\delta}(t)y+\int_{0}^{t}P_{\gamma,\delta}(t-s)\eta(s)ds$ is strong solution of the problem \eqref{eq5.1}-\eqref{eq5.2} for $\eta(t,\omega(t)) = \eta(t)$ if $x,y\in D(\mathcal{A})$ and there exist a function $
\phi \in L^{1}((0,T),X) $ such that $\eta(t)=J_{t}^{\delta(2-\gamma)}\phi(t), \phi(t) \in D(\mathcal{A})$ and $\mathcal{A}\phi(t) \in L^{1}$.
\end{thm}
\begin{proof}
We can write,
\begin{eqnarray*}
&&\int_{0}^{t}P_{\gamma,\delta}(t-s)\eta(s)ds=P_{\gamma,\delta}(t)*\eta(t)=J_{t}^{1-\delta(2-\gamma)}C_{\gamma,\delta}(t)*\eta(t)\\
&&=g_{1-\delta(2-\gamma)}(t)*(C_{\gamma,\delta}(t)*\eta(t)).
\end{eqnarray*}
\begin{eqnarray*}
&&J_{t}^{(1-\delta)(2-\gamma)}(P_{\gamma,\delta}(t)*\eta(t))=J_{t}^{1+(2-\gamma)(1-2\delta)}C_{\gamma,\delta}(t)*\eta(t)\\
&&=J_{t}^{1+(2-\gamma)(1-2\delta)}\bigg[\frac{t^{\gamma+\delta(2-\gamma)-2}}{\Gamma(\gamma+\delta(2-\gamma)-1}+\mathcal{A}J_{t}^{\gamma}C_{\gamma,\delta}(t)\bigg]*\eta(t)\\
&&=J_{t}^{1+(2-\gamma)(1-2\delta)}\bigg[\frac{t^{\gamma+\delta(2-\gamma)-2}}{\Gamma(\gamma+\delta(2-\gamma)-1}*\eta(t)+\mathcal{A}J_{t}^{\gamma}C_{\gamma,\delta}(t)*\eta(t)\bigg]\\
&&=g_{1+(2-\gamma)(1-2\delta)}(t)*\bigg(\frac{t^{\gamma+\delta(2-\gamma)-2}}{\Gamma(\gamma+\delta(2-\gamma)-1}*\eta(t)\bigg)+g_{1+(2-\gamma)(1-2\delta)}(t)*\bigg(\mathcal{A}J_{t}^{\gamma}C_{\gamma,\delta}(t)*\eta(t)\bigg)\\
&&=J_{t}^{2-\delta(2-\gamma)}\eta(t)+J_{t}^{3-2\delta(2-\gamma)}C_{\gamma,\delta}(t)*\mathcal{A}\eta(t)
\end{eqnarray*}
\begin{eqnarray*}
&&\frac{d}{dt}J_{t}^{(1-\delta)(2-\gamma)}(P_{\gamma,\delta}(t)*\eta(t))\\
&&=J_{t}^{1-\delta(2-\gamma)}\eta(t)+J_{t}^{2-2\delta(2-\gamma)}C_{\gamma,\delta}(t)*\mathcal{A}\eta(t)\\
&&=J_{t}\phi(t)+J_{t}^{2-2\delta(2-\gamma)}C_{\gamma,\delta}(t)*\mathcal{A}J_{t}^{\delta(2-\gamma)}\phi(t)\\
&&=J_{t}\phi(t)+J_{t}^{2-\delta(2-\gamma)}C_{\gamma,\delta}(t)*\mathcal{A}\phi(t).
\end{eqnarray*}
\begin{eqnarray*}
&&\frac{d^{2}}{dt^{2}}J_{t}^{(1-\delta)(2-\gamma)}(P_{\gamma,\delta}(t)*\eta(t))\\
&&=\phi(t)+J_{t}^{1-\delta(2-\gamma)}C_{\gamma,\delta}(t)*\mathcal{A}\phi(t)=\phi(t)+P_{\gamma,\delta}(t)*\mathcal{A}\phi(t)
\end{eqnarray*}
\begin{eqnarray*}
&&J_{t}^{\delta(2-\gamma)}\frac{d^{2}}{dt^{2}}J_{t}^{(1-\delta)(2-\gamma)}(P_{\gamma,\delta}(t)*\eta(t))\\
&&=\eta(t)+\mathcal{A}(P_{\gamma,\delta}(t)*\eta(t)).
\end{eqnarray*}
Thus $D_{0}^{\gamma,\delta}\omega(t)=\mathcal{A}\omega(t)+\eta(t)$.
\end{proof}

We assume some hypothesis to deal with the existence results for the problem \eqref{eq5.1}-\eqref{eq5.2}. They are following:
\begin{description}
\item[H1] $\lVert C_{\gamma,\delta}(t) \rVert \leq M t^{\gamma+\delta(2-\gamma)-2},\; t >0$ for a constant $M>0$;
	\item[H2] $P_{\gamma,\delta}(t)(t>0)$ is continuous in the uniform operator topology.
	\item[H3] For each $t\in I'$, $\eta(t,\cdot):X \rightarrow X$ is continuous and for each $x\in X$,  $\eta(\cdot,x):I'\rightarrow X$ is strongly measurable;
	\item[H4]
	$\lVert \eta(t,x)\rVert \leq m(t)$ for a $m\in L(I',\mathbb{R}^{+})$ and all $x \in B_{r}^{Y}(I')$ at a.e.  $t\in [0,T]$;
	\item[H5] there exist a constant $r>0$ such that 
	$$M\lVert \omega_{1}\rVert+\frac{M}{\gamma+\delta(2-\gamma)-1}\lVert \omega_{2}\rVert +\frac{M\Gamma(\gamma+\delta(2-\gamma)-1) T^{1-\delta(2-\gamma)}}{\Gamma(\gamma)}\lVert m \rVert_{L^{1}}=r$$
	\end{description}
\begin{lemma}Assume that $(H1)$ holds, then 
$\lVert S_{\gamma,\delta}(t) \rVert \leq \frac{M}{\gamma+\delta(2-\gamma)-1} t^{\gamma+\delta(2-\gamma)-1}$ and $\lVert P_{\gamma,\delta}(t) \rVert \leq \frac{M\Gamma(\gamma+\delta(2-\gamma)-1)}{\Gamma(\gamma)}t^{\gamma-1}$ for $t>0$.
\end{lemma}
\begin{proof}
Since $S_{\gamma,\delta}(t)=\int_{0}^{t}C_{\gamma,\delta}(s)ds$ and $P_{\gamma,\delta}(t)=J_{t}^{1-\delta(2-\gamma)}C_{\gamma,\delta}$, using Hypothesis (H1) we have
$$\lVert S_{\gamma,\delta}(t) \rVert \leq M\int_{0}^{t}s^{\gamma+\delta(2-\gamma)-2}ds \leq \frac{M t^{\gamma+\delta(2-\gamma)-1}}{\gamma+\delta(2-\gamma)-1}$$
and
\begin{eqnarray*}
\lVert P_{\gamma,\delta}(t) \rVert &=&\bigg\lVert \frac{1}{\Gamma(1-\delta(2-\gamma))}\int_{0}^{t}(t-s)^{-\delta(2-\gamma)}C_{\gamma,\delta}(s)ds \bigg\rVert\\
& \leq &\frac{M}{\Gamma(1-\delta(2-\gamma))}\int_{0}^{t}(t-s)^{-\delta(2-\gamma)} s^{\gamma+\delta(2-\gamma)-2}ds\\
&=& \frac{M\Gamma(\gamma+\delta(2-\gamma)-1)}{\Gamma(\gamma)}t^{\gamma-1}.
\end{eqnarray*}
\end{proof}

For any $u\in B_{r}^{Y}$, we define an operator $S$ as follows
$$ (Su)(t)=(S_{1}u)(t)+(S_{2}u)(t)$$
where
$$ (S_{1}u)(t)=C_{\gamma,\delta}(t)\omega_{1}+S_{\gamma,\delta}(t)\omega_{2}, \; \mbox{for}\; t\in (0,T],$$
$$ (S_{2}u)(t)=\int_{0}^{t}P_{\gamma,\delta}(t-s)\eta(s,\omega(s))ds , \; \mbox{for}\; t\in (0,T].$$
We can see easily that $\lim_{t\rightarrow 0+}t^{2-\gamma-\delta(2-\gamma)}(S_{1}u)(t)=\frac{\omega_{1}}{\Gamma(\gamma+\delta(2-\gamma)-1)}$ and  $\lim_{t\rightarrow 0+}t^{2-\gamma-\delta(2-\gamma)}(S_{2}u)(t)=0$. For $y\in B_{r}(I)$, set $\omega(t)=t^{\gamma+\delta(2-\gamma)-2}y(t)$ for $t\in (0, T]$. Then $u \in B_{r}^{Y}$.\\
We define an operator $\mathfrak{S}$ as follows
$$ (\mathfrak{S}y)(t)=(\mathfrak{S}_{1}y)(t)+(\mathfrak{S}_{2}y)(t),$$
where\\
$ (\mathfrak{S}_{1}y)(t)=
\begin{array}{cc}
\bigg\{ & 
\begin{array}{cc}
t^{2-\gamma-\delta(2-\gamma)}(S_{1}u)(t) & t\in (0,T]\\
\frac{\omega_{1}}{\Gamma(\gamma+\delta(2-\gamma)-1)} & \mbox{for}\; t=0,
\end{array}
\end{array}
$
\begin{center}
and
\end{center}
$ (\mathfrak{S}_{2}y)(t)=
\begin{array}{cc}
\bigg\{ & 
\begin{array}{cc}
t^{2-\gamma-\delta(2-\gamma)}(S_{2}u)(t) & t\in (0,T]\\
0 & \mbox{for}\; t=0,
\end{array}
\end{array}
$
\begin{lemma}
For $t>0$, $S_{\gamma,\delta}(t)$ and $P_{\gamma,\delta}(t)$ are strongly continuous.
\end{lemma}
\begin{proof}
For $x\in X$  and $0<t_{1}<t_{2}\leq T$, we have
\begin{eqnarray*}
&&\lVert S_{\gamma,\delta}(t_{2})x-S_{\gamma,\delta}(t_{1})x\rVert = \bigg\lVert \int_{t_{1}}^{t_{2}}C_{\gamma,\delta}(s)xds \bigg\rVert \leq M\int_{t_{1}}^{t_{2}}s^{\gamma+\delta(2-\gamma)-2}\lVert x\rVert ds\\
&& = \frac{M}{\Gamma(\gamma+\delta(2-\gamma))}(t_{2}^{\gamma+\delta(2-\gamma)-1}-t_{1}^{\gamma+\delta(2-\gamma)-1}) \rightarrow 0\; \mbox{as} \;t_{2} \rightarrow t_{1},
\end{eqnarray*}
\begin{eqnarray*}
&&\lVert P_{\gamma,\delta}(t_{2})x-P_{\gamma,\delta}(t_{1})x\rVert= \\
&& \bigg\lVert \frac{1}{\Gamma(1-\delta(2-\gamma))}\int_{0}^{t_{2}}(t_{2}-s)^{-\delta(2-\gamma)}C_{\gamma,\delta}(s)xds-\frac{1}{\Gamma(1-\delta(2-\gamma))}\int_{0}^{t_{1}}(t_{1}-s)^{-\delta(2-\gamma)}C_{\gamma,\delta}(s)xds  \bigg\rVert\\
&& \leq \bigg\lVert \frac{1}{\Gamma(1-\delta(2-\gamma))}\int_{t_{1}}^{t_{2}}(t_{2}-s)^{-\delta(2-\gamma)}C_{\gamma,\delta}(s)xds\bigg \rVert + \\
&& \bigg\lVert \frac{1}{\Gamma(1-\delta(2-\gamma))}\int_{0}^{t_{1}}((t_{2}-s)^{-\delta(2-\gamma)}-(t_{1}-s)^{-\delta(2-\gamma)})C_{\gamma,\delta}(s)xds \bigg \rVert\\
&& := I_{1}+I_{2}
\end{eqnarray*}
\begin{eqnarray*}
&& I_{1}\leq \frac{M}{\Gamma(1-\delta(2-\gamma))\Gamma(\gamma+\delta(2-\gamma)-1)}\int_{t_{1}}^{t_{2}}(t_{2}-s)^{-\delta(2-\gamma)}s^{\gamma+\delta(2-\gamma)-2}\lVert x \rVert ds\\
&& \leq \frac{M t_{1}^{\gamma+\delta(2-\gamma)-2}}{\Gamma(1-\delta(2-\gamma))\Gamma(\gamma+\delta(2-\gamma)-1)}\int_{t_{1}}^{t_{2}}(t_{2}-s)^{-\delta(2-\gamma)}\lVert x \rVert ds \\
&&=\frac{M t_{1}^{\gamma+\delta(2-\gamma)-2}}{\Gamma(1-\delta(2-\gamma))\Gamma(\gamma+\delta(2-\gamma)-1)}\frac{(t_{2}-t_{1})^{1-\delta(2-\gamma)}}{1-\delta(2-\gamma)} \rightarrow 0\; \mbox{as} \;t_{2} \rightarrow t_{1}.     
\end{eqnarray*}
\begin{eqnarray*}
&& I_{2} \leq \frac{M}{\Gamma(1-\delta(2-\gamma))\Gamma(\gamma+\delta(2-\gamma)-1)}\int_{0}^{t_{1}}|(t_{2}-s)^{-\delta(2-\gamma)}-(t_{1}-s)^{-\delta(2-\gamma)}|s^{\gamma+\delta(2-\gamma)-2}\lVert x \rVert ds
\end{eqnarray*}
Noting that $|(t_{2}-s)^{-\delta(2-\gamma)}-(t_{1}-s)^{-\delta(2-\gamma)}|s^{\gamma+\delta(2-\gamma)-2}\leq 2(t_{1}-s)^{-\delta(2-\gamma)}s^{\gamma+\delta(2-\gamma)-2}$ and $\int_{0}^{t_{1}}(t_{1}-s)^{-\delta(2-\gamma)}s^{\gamma+\delta(2-\gamma)-2}ds$ exists, then by Lebesgue's dominated convergence theorem, we have $I_{2}  \rightarrow 0$ as $t_{2} \rightarrow t_{1}$.

\end{proof}

\begin{lemma}\label{lemma7}
Assume that (H1)-(H4) hold, then $\{\mathfrak{S}y: y\in B_{r}(I)\}$ is equicontinuous.
\end{lemma}
\begin{proof}
Let $t_{1}=0$ and $0 <t_{2}\leq T$,
\begin{eqnarray*}
&& \lVert \mathfrak{S}y(t_{2})-\mathfrak{S}y(0)\rVert \leq \lVert \mathfrak{S_{1}}y(t_{2})-\mathfrak{S_{1}}y(0)\rVert +\lVert \mathfrak{S_{2}}y(t_{2})-\mathfrak{S_{2}}y(0)\rVert\\
&& \rightarrow 0\; \mbox{as} \;t_{2} \rightarrow t_{1}
\end{eqnarray*}
Now, for $0<t_{1}<t_{2} \leq T$, we have
\begin{eqnarray*}
&& \lVert \mathfrak{S}y(t_{2})-\mathfrak{S}y(t_{1})\rVert \leq \lVert t_{2}^{2-\gamma-\delta(2-\gamma)}C_{\gamma,\delta}(t_{2})\omega_{1}-t_{1}^{2-\gamma-\delta(2-\gamma)}C_{\gamma,\delta}(t_{1})\omega_{1}\rVert\\
&& + \lVert t_{2}^{2-\gamma-\delta(2-\gamma)}S_{\gamma,\delta}(t_{2})\omega_{2}-t_{1}^{2-\gamma-\delta(2-\gamma)}S_{\gamma,\delta}(t_{1})\omega_{2}\rVert+ \\
&& \lVert t_{2}^{2-\gamma-\delta(2-\gamma)}\int_{0}^{t_{2}}P_{\gamma,\delta}(t_{2}-s)\eta(s,\omega(s))ds-t_{1}^{2-\gamma-\delta(2-\gamma)}\int_{0}^{t_{1}}P_{\gamma,\delta}(t_{1}-s)\eta(s,\omega(s))ds\rVert\\
&& := I_{1}+I_{2}+I_{3}.
\end{eqnarray*}
From strong continuity of $C_{\gamma,\delta}(t)$ and $S_{\gamma,\delta}(t)$, we get $I_{1}, I_{2} \rightarrow 0$ as $t_{2} \rightarrow t_{1}$.	
\begin{eqnarray*}
&& I_{3}\leq \lVert t_{2}^{2-\gamma-\delta(2-\gamma)}\int_{t_{1}}^{t_{2}}P_{\gamma,\delta}(t_{2}-s)\eta(s,\omega(s))ds\rVert +\\
&& \lVert t_{2}^{2-\gamma-\delta(2-\gamma)}\int_{0}^{t_{1}}P_{\gamma,\delta}(t_{2}-s)\eta(s,\omega(s))ds- t_{1}^{2-\gamma-\delta(2-\gamma)}\int_{0}^{t_{1}}P_{\gamma,\delta}(t_{2}-s)\eta(s,\omega(s))ds \rVert\\
&& + \lVert t_{1}^{2-\gamma-\delta(2-\gamma)}\int_{0}^{t_{1}}P_{\gamma,\delta}(t_{2}-s)\eta(s,\omega(s))ds- t_{1}^{2-\gamma-\delta(2-\gamma)}\int_{0}^{t_{1}}P_{\gamma,\delta}(t_{1}-s)\eta(s,\omega(s))ds\rVert\\
&& :=I_{31}+I_{32}+I_{33}.
\end{eqnarray*}
Using the Hypothesis (H4), we have
\begin{eqnarray*}
&& I_{31} \leq M t_{2}^{2-\gamma-\delta(2-\gamma)}\int_{t_{1}}^{t_{2}}(t_{2}-s)^{\gamma-1}m(s)ds \leq M t_{2}^{2-\gamma-\delta(2-\gamma)}(t_{2}-t_{1})^{\gamma-1}\int_{t_{1}}^{t_{2}}m(s)ds\\
&& \quad \rightarrow 0 \; \mbox{as}\; t_{2} \rightarrow t_{1}.
\end{eqnarray*}
\begin{eqnarray*}
&& I_{32} \leq M |t_{2}^{2-\gamma-\delta(2-\gamma)}-t_{1}^{2-\gamma-\delta(2-\gamma)}|\int_{0}^{t_{1}}(t_{2}-s)^{\gamma-1}m(s)ds\\ && \leq |t_{2}^{2-\gamma-\delta(2-\gamma)}-t_{1}^{2-\gamma-\delta(2-\gamma)}|t_{2}^{\gamma-1}\int_{0}^{t_{1}}m(s)ds\\
&& \quad \rightarrow 0 \; \mbox{as}\; t_{2} \rightarrow t_{1}.
\end{eqnarray*}
Using Hypotheses (H2) and (H4), we have
\begin{eqnarray*}
&& I_{33} \leq t_{1}^{2-\gamma-\delta(2-\gamma)}\int_{0}^{t_{1}}\lVert P_{\gamma,\delta}(t_{2}-s)-P_{\gamma,\delta}(t_{1}-s)\rVert m(s)ds\\
&& =t_{1}^{2-\gamma-\delta(2-\gamma)}\bigg(\int_{0}^{t_{1}-\epsilon}\lVert P_{\gamma,\delta}(t_{2}-s)-P_{\gamma,\delta}(t_{1}-s)\rVert m(s)ds+\int_{t_{1}-\epsilon}^{t_{1}}\lVert P_{\gamma,\delta}(t_{2}-s)-P_{\gamma,\delta}(t_{1}-s)\rVert m(s)ds\bigg)\\
&& \leq t_{1}^{2-\gamma-\delta(2-\gamma)}\sup_{s\in  [0,t_{1}-\epsilon]}\lVert P_{\gamma,\delta}(t_{2}-s)-P_{\gamma,\delta}(t_{1}-s)\rVert \int_{0}^{t_{1}}m(s)ds+t_{1}^{2-\gamma-\delta(2-\gamma)}C\int_{t_{1}-\epsilon}^{t_{1}}m(s)ds
\end{eqnarray*}
$\rightarrow 0$ as $t_{2} \rightarrow t_{1}$.
\end{proof}
\begin{lemma}
Assume that (H1)-(H5) hold, then $\mathfrak{S}$ maps $B_{r}(I)$ into $B_{r}(I)$, and  $\mathfrak{S}$ is continuous in $B_{r}(I)$.
\end{lemma}
\begin{proof}
Let $w \in B_{r}^{Y}(I') \implies \lVert w\rVert_{Y} \leq r$, clearly $S\omega(t) \in C(I',X)$ and
\begin{eqnarray*}
	&& t^{(2-\gamma)(1-\delta)}\lVert S\omega(t)\rVert \leq t^{(2-\gamma)(1-\delta)}(\lVert C_{\gamma,\delta}(t)\omega_{1}\rVert+\lVert S_{\gamma,\delta}(t)\omega_{1}\rVert+\int_{0}^{t}\lVert P_{\gamma,\delta}(t)\rVert m(s)ds)\\
	&& \leq M\lVert \omega_{1}\rVert+\frac{M}{\gamma+\delta(2-\gamma)-1}\lVert \omega_{2}\rVert +\frac{M \Gamma(\gamma+\delta(2-\gamma)-1)t^{(2-\gamma)(1-\delta)}}{\Gamma(\gamma)}\int_{0}^{t}(t-s)^{\gamma-1}m(s)ds\\
	&& \leq M\lVert \omega_{1}\rVert+\frac{M}{\gamma+\delta(2-\gamma)-1}\lVert \omega_{2}\rVert +\frac{M \Gamma(\gamma+\delta(2-\gamma)-1)t^{(2-\gamma)(1-\delta)}t^{\gamma-1}}{\Gamma(\gamma)}\int_{0}^{t}m(s)ds\\
	&& \leq \frac{M}{\Gamma(\gamma+\delta(2-\gamma)-1)}\lVert \omega_{1}\rVert+\frac{M}{\Gamma(\gamma+\delta(2-\gamma))}\lVert \omega_{2}\rVert +\frac{M \Gamma(\gamma+\delta(2-\gamma)-1)T^{1-\delta(2-\gamma)}}{\Gamma(\gamma)}\lVert m \rVert_{L^{1}}=r
\end{eqnarray*}
Hence $\mathfrak{S}$ maps $B_{r}(I)$ into $B_{r}(I)$.
Now, we prove that  $\mathfrak{S}$ is continuous in $B_{r}(I)$.\\
Let $y_{m}, y \in B_{r}(I)$ such that $\lim_{m\rightarrow \infty}y_{m}=y$, then we have $\lim_{m\rightarrow \infty}y_{m}(t)=y(t)$ and $ \lim_{m\rightarrow \infty}w_{m}(t)=\omega(t)$, for $t \in (0,T]$ where $w_{m}(t)=t^{\gamma+\delta(2-\gamma)-2}y_{m}(t)$ and $\omega(t)=t^{\gamma+\delta(2-\gamma)-2}y(t)$.\\
Now for $t\in [0,T]$ 
\begin{eqnarray*}
&& \lVert (\mathfrak{S}y_{m})(t)-(\mathfrak{S}y)(t)\rVert= t^{2-\gamma-\delta(2-\gamma)}\bigg\lVert \int_{0}^{t} P_{\gamma,\delta}(t-s)(\eta(s,w_{m}(s))-\eta(s,\omega(s)) ds\bigg\rVert \\
&& \leq t^{2-\gamma-\delta(2-\gamma)}\frac{M}{\Gamma(\gamma)}\int_{0}^{t}(t-s)^{\gamma-1}\lVert \eta(s,w_{m}(s))-\eta(s,\omega(s)\rVert ds.
\end{eqnarray*}
$(t-s)^{\gamma-1}\lVert \eta(s,w_{m}(s))-\eta(s,\omega(s)\rVert \leq 2T^{\gamma-1}m(s)$, which is integrable. Hence by Lebesuge dominated convergence theorem, we have $\lVert (\mathfrak{S}y_{m})(t)-(\mathfrak{S}y)(t)\rVert \rightarrow 0$. But Lemma \ref{lemma7} ensures the continuity of $\mathfrak{S}$.
\end{proof}
\begin{description}
\item[H6] $\psi(\eta(t,\mathcal{D}))\leq k\psi(\mathcal{D})$, a.e. $t \in [0,T]$, for any bounded subset $\mathcal{D}$ of $X$.
\end{description}
\begin{thm}\label{theorem5}
Under the assumption (H1)-(H6), the Cauchy problem \eqref{eq5.1}-\eqref{eq5.2} has at least one mild solution in $B_{r}^{Y}(I')$.
\end{thm}
\begin {proof}
Let $\mathfrak{B}_{0}\subset B_{r}(I)$\\
Define $\mathfrak{S}^{(1)}(B_{0})=\mathfrak{S}(B_{0}), \mathfrak{S}^{(n)}(\mathfrak{B}_{0})=\mathfrak{S}(\overline{co}(\mathfrak{S}^{(n-1)}(\mathfrak{B}_{0}))), n=2,3 \ldots$.\\
Making use of proposition  \ref{prop2.7}-\ref{prop2.9}, we are able to get a sequence $\{y_{n}^{(1)}\}$ in $\mathfrak{B}_{0}$ for any $\epsilon>0$ such that
\begin{eqnarray*}
        &&\psi(\mathfrak{S}^{(1)}(\mathfrak{B}_{0}(t)))\\ 
&&\leq2\psi\bigg(t^{(1-\delta)(2-\gamma)}\int_{0}^{t}P_{\gamma,\delta}(t-s)\eta(s,\{s^{-(2-\gamma)(1-\delta)}y_{n}^{(1)}(s)\}_{n=1}^{\infty})ds\bigg)+\epsilon\\
  &&      \leq
4M_{\gamma,\delta}t^{(2-\gamma)(1-\delta)}\int_{0}^{t}(t-s)^{\gamma-1}\psi(\eta(s,\{s^{-(2-\gamma)(1-\delta)}y_{n}^{(1)}(s)\}_{n=1}^{\infty}))ds+\epsilon\\
    &&    \leq
4M_{\gamma,\delta}kt^{(2-\gamma)(1-\delta)}\psi(\mathfrak{B}_{0})\int_{0}^{t}(t-s)^{\gamma-1}s^{(\gamma-2)(1-\delta)}ds+\epsilon\\
   &&     =4M_{\gamma,\delta}kt^{\gamma}\psi(\mathfrak{B}_{0})\frac{\Gamma(\gamma)\Gamma((\gamma+\delta(2-\gamma)-1)}{\Gamma(2\gamma+\delta(2-\gamma)-1)}+\epsilon.
                \end{eqnarray*}
       
                Since $\epsilon >0$ is arbitrary, we have
                \begin{equation*}
                        \psi(\mathfrak{S}^{(1)}(\mathfrak{B}_{0}(t)))\leq4M_{\gamma,\delta}kt^{\gamma}\psi(\mathfrak{B}_{0})\frac{\Gamma(\gamma)\Gamma((\gamma+\delta(2-\gamma)-1)}{\Gamma(2\gamma+\delta(2-\gamma)-1)}.
\end{equation*}
Again using proposition  \ref{prop2.7}-\ref{prop2.9}, we are able to get a sequence $\{y_{n}^{(2)}\}$ in $\overline{co}(B_{0})$ for any $\epsilon>0$ such that
 \begin{eqnarray*}
    &&    \psi(\mathfrak{S}^{(2)}(\mathfrak{B}_{0}(t)))=\psi(\mathfrak{S}(\overline{co}(\mathfrak{S}^{(1)}(\mathfrak{B}_{0}(t)))))\\
&&\leq
2\psi\bigg(t^{(2-\gamma)(1-\delta)}\int_{0}^{t}P_{\gamma,\delta}(t-s)\eta(s,\{s^{-(2-\gamma)(1-\delta)}y_{n}^{(2)}(s)\}_{n=1}^{\infty})ds\bigg)+\epsilon\\
&&  \leq
4M_{\gamma,\delta}t^{(2-\gamma)(1-\delta)}\int_{0}^{t}(t-s)^{\gamma-1}\psi(\eta(s,\{s^{-(2-\gamma)(1-\delta)}y_{n}^{(2)}(s)\}_{n=1}^{\infty}))ds+\epsilon\\
&& \leq
4M_{\gamma,\delta}kt^{(2-\gamma)(1-\delta)}\int_{0}^{t}(t-s)^{\gamma-1}\psi(\{s^{-(2-\gamma)(1-\delta)}y_{n}^{(2)}(s)\}_{n=1}^{\infty})ds+\epsilon\\
&&\leq
4M_{\gamma,\delta}kt^{(2-\gamma)(1-\delta)}\int_{0}^{t}(t-s)^{\gamma-1}s^{-(2-\gamma)(1-\delta)}\psi(\{y_{n}^{(2)}(s)\}_{n=1}^{\infty})ds+\epsilon\\
&&\leq\frac{(4M_{\gamma,\delta}k)^{2}t^{(2-\gamma)(1-\delta)}\Gamma(\gamma)\Gamma(\gamma+\delta(2-\gamma)-1)}{\Gamma(2\gamma+\delta(2-\gamma)-1)}\psi(\mathfrak{B}_{0})\int_{0}^{t}(t-s)^{\gamma-1}s^{2\gamma+\delta(2-\gamma)-2}ds+\epsilon\\
&&=\frac{(4M_{\gamma,\delta}k)^{2}t^{2\gamma}\Gamma^{2}(\gamma)\Gamma(\gamma+\delta(2-\gamma)-1)}{\Gamma(3\gamma+\delta(2-\gamma)-1)}\psi(\mathfrak{B}_{0})+\epsilon.
                \end{eqnarray*}
       
               We can prove by mathematical induction that
                \begin{equation*}
                        \psi(\mathfrak{S}^{(n)}(\mathfrak{B}_{0}(t)))\leq
\frac{(4kM_{\gamma,\delta})^{n}t^{n\gamma}\Gamma^{n}(\gamma)\Gamma(\gamma+\delta(2-\gamma)-1)}{\Gamma((n+1)\gamma+\delta(2-\gamma)-1)}\psi(\mathfrak{B}_{0}),\; n \in
\mathbb{N}.
                \end{equation*}
                Since \\
                \indent $\lim_{n\rightarrow
\infty}\frac{(4kM_{\gamma,\delta})^{n}t^{n\gamma}\Gamma^{n}(\gamma)\Gamma(\gamma+\delta(2-\gamma)-1)}{\Gamma((n+1)\gamma+\delta(2-\gamma)-1)}=0$,\\
                there exists a constant $n_{0}$ such that
                \begin{eqnarray*}
                        &&\frac{(4kM_{\gamma,\delta})^{n}t^{n_{0}\gamma}\Gamma^{n_{0}}(\gamma)\Gamma(\gamma+\delta(2-\gamma)-1)}{\Gamma((n_{0}+1)\gamma+\delta(2-\gamma)-1)}\\
                        && \quad \leq \frac{(4kM_{\gamma,\delta})^{n}T^{n_{0}\gamma}\Gamma^{n_{0}}(\gamma)\Gamma(\gamma+\delta(2-\gamma)-1)}{\Gamma((n_{0}+1)\gamma+\delta(2-\gamma)-1)}=p<1.
                \end{eqnarray*}
                Then\\
                \indent $\psi(\mathfrak{T}^{(n_{0})}(\mathfrak{B}_{0}(t)))\leq p\psi(\mathfrak{B}_{0}).$\\
             Since 
$\mathfrak{T}^{(n_{0})}(\mathfrak{B}_{0}(t))$ is equicontinuous and bounded (see- Proposition 1.26 of \cite{YZ},  using Proposition \ref{prop2.7}, we have\\
                \indent
$\psi(\mathfrak{T}^{n_{0}}(\mathfrak{B}_{0}))=\max_{t\in[0,T]}\psi(\mathfrak{T}^{n_{0}}(\mathfrak{B}_{0}(t))).$\\
                Hence,\\
                \indent  $\psi(\mathfrak{T}^{n_{0}}(\mathfrak{B}_{0}))\leq p\psi(\mathfrak{B}_{0}).$'
                theorem gives a fixed point $y^{*}$
in $\mathcal{B}_{r}(I)$ of $\mathfrak{S}$ . Let $w^{*}(t)=t^{(2-\gamma)(\delta-1)}y^{*}(t)$. Then
$w^{*}(t)$ is a mild solution of  \eqref{eq5.1}-\eqref{eq5.2}.
\end{proof}
\begin{description}
\item[H7]  For any $w, v\in B_{r}^{Y}(I')$ and $t \in (0,T]$ we have
$$\lVert \eta(t,\omega(t))-\eta(t,v(t))\rVert \leq k \lVert w-v\rVert_{Y},$$
where  $k>0$ is a constant.
\end{description}
\begin{thm}
Assume that $(H1),(H3)-(H5)$  and $(H7)$ holds. Then the problem   \eqref{eq5.1}-\eqref{eq5.2} has a unique mild solution for every $\omega_{1},\omega_{2} \in X$, if $\frac{M}{\Gamma(\gamma)}T^{2-\delta(2-\gamma)}k<1$
\end{thm}
\begin{proof}
Let $y_{1}(t)=t^{2-\gamma-\delta(2-\gamma)}\omega(t)$ and $y_{2}^(t)=t^{2-\gamma-\delta(2-\gamma)}v(t)$.
\begin{eqnarray*}
	&& \lVert (\mathfrak{S}y_{1})(t)-(\mathfrak{S}y_{2})(t)\rVert= t^{2-\gamma-\delta(2-\gamma)}\lVert \int_{0}^{t} P_{\gamma,\delta}(t-s)(\eta(s,\omega(s))-\eta(s,v(s))\rVert\\
	&& \leq t^{2-\gamma-\delta(2-\gamma)}\frac{M}{\Gamma(\gamma)}\int_{0}^{t}(t-s)^{\gamma-1}\lVert \eta(s,\omega(s))-\eta(s,v(s)\rVert\\
	&& \leq t^{2-\delta(2-\gamma)}\frac{M}{\Gamma(\gamma)}k \lVert w-v\rVert_{Y}\\
	&& \leq T^{2-\delta(2-\gamma)}\frac{M}{\Gamma(\gamma)}k \lVert y_{1}-y_{2}\rVert
\end{eqnarray*}
Hence unique mild solution exist by Banach contraction theorem.
\end{proof}
\section{Example}
Let $X=L^{2}([0,\pi],\mathbb{R})$. We consider a FPDE with Hilfer fractional derivative
\begin{equation}\label{exhilfer2eq1}
D_{0}^{\gamma,\delta}\omega(x,t)=\frac{\partial^{2}}{\partial x^{2}}\omega(x,t)+\eta(t,\omega(x,t))
\end{equation}
\begin{equation}\label{exhilfer2eq2}
\omega(0,t)=0=\omega(\pi,t)
\end{equation}
\begin{equation}\label{exhilfer2eq3}
(g^{(1-\delta)(2-\gamma)}*\omega(x,t))(t=0)=\omega_{1}(x), (g^{(1-\delta)(2-\gamma)}*\omega)^{'}(t=0)=\omega_{2}(x)
\end{equation}
where $\eta$ is a given function.\\
To convert the problem \eqref{exhilfer2eq1}-\eqref{exhilfer2eq3} into a abstract setup in $X$, we  define an operator $\mathcal{A}$ by $\mathcal{A}w=w^{\prime \prime}$ with the domain $$D(\mathcal{A})=\{w()\in X: w,w^{\prime}\;\mbox{is absolutely continuous and}\; w^{\prime \prime}\in X, \mbox{with}\;\omega(0)=w(\pi)=0\}.$$
$\mathcal{A}$ has eigen values $\{\gamma_{n}=-n^{2}\}$ with eigen functions $\{\sin nx, n=1,2, \ldots\}.$  
We define the family of operators $\{C_{\gamma,\delta}(t)\}, \{S_{\gamma,\delta}(t)\}$ as follows
$$(C_{\gamma,\delta}(t)g)(x)=\sum_{n=1}^{\infty} t^{\gamma+\delta(2-\gamma)-2}E_{\gamma,\gamma+\delta(2-\gamma)-1}(-n^{2}t^{\gamma})g_{n}\sin nx,$$
$$(S_{\gamma,\delta}(t)g)(x)=\sum_{n=1}^{\infty} t^{\gamma+\delta(2-\gamma)-1}E_{\gamma,\gamma+\delta(2-\gamma)}(-n^{2}t^{\gamma})g_{n}\sin nx,$$
where $E_{\gamma,\gamma}(\cdot)$ is the Mittag-Leffler function.  \\
The operator $\mathcal{A}$ is a sectorial operator of type $(M,\theta,\gamma,\mu)$. We can easily calculate to get a constant $M>0$ such that $\lVert C_{\gamma,\delta}(t)\rVert \leq M t^{\gamma+\delta(2-\gamma)-2}$. It is well known from \cite{LPP} that $P_{\gamma,\delta}(t)$ is continuous in the uniform operator topology. Consider $\eta(t,\omega)=t\sin(\omega)$. Then $r=M\lVert \omega_{1}\rVert+\frac{M}{\gamma+\delta(2-\gamma)-1}\lVert \omega_{2} \rVert|\frac{M\Gamma(\gamma+\delta(2-\gamma))}{2\Gamma(\gamma)}$. Then Theorem \ref{theorem5} guarantees the existence of a mild solution.

\end{document}